\documentclass[12pt]{amsart}

\usepackage[colorlinks]{hyperref}
\usepackage{fullpage}
\usepackage{stmaryrd}
\usepackage{graphicx}
\usepackage{hyperref}
\usepackage{amsmath,amssymb,amsfonts,mathrsfs,amscd}
\usepackage{newtxtext,newtxmath}
\usepackage[all,cmtip]{xy}

\newtheorem{thm}{Theorem}[section]
\newtheorem{lem}[thm]{Lemma}
\newtheorem{cor}[thm]{Corollary}

\newtheorem{thmA}{Theorem}

\newtheorem{corA}[thmA]{Corollary}

\theoremstyle{definition}
\newtheorem{defn}[thm]{Definition}

\theoremstyle{remark}

\numberwithin{equation}{section}

\DeclareMathOperator{\B}{{\rm B}_{\pi}}

\DeclareMathOperator{\N}{{\rm N}_{\pi}}
\DeclareMathOperator{\D}{{\rm D}_{\pi}}
\DeclareMathOperator{\I}{{\rm I}_\pi}

\newcommand{\NM}{\vartriangleleft}

\DeclareMathOperator{\Irr}{Irr}
\DeclareMathOperator{\IBr}{IBr}

\begin{document}

\title{The uniqueness of vertex pairs in $\pi$-separable groups}

\author{Lei Wang}
\address{School of Mathematical Sciences, Shanxi University, Taiyuan, 030006, China.}
\email{wanglei0115@163.com}

\author{Ping Jin*}
\address{School of Mathematical Sciences, Shanxi University, Taiyuan, 030006, China.}
\email{jinping@sxu.edu.cn}

\thanks{*Corresponding author}


\date{}

\maketitle

\begin{abstract}
Let $G$ be a finite $\pi$-separable group, where $\pi$ is a set of primes,
and let $\chi$ be an irreducible complex character that is a $\pi$-lift of some $\pi$-partial character of $G$.
It was proved by Cossey and Lewis that all of the vertex pairs for $\chi$ are linear and conjugate in $G$ if $2\in\pi$, but the result can fail for $2\notin\pi$.
In this paper we introduce the notion of the twisted vertices in the case where $2\notin\pi$, and establish the uniqueness for linear twisted vertices under the conditions that either $\chi$ is an $\mathcal N$-lift for a $\pi$-chain $\mathcal N$ of $G$ or it has a linear Navarro vertex, thus answering a question proposed by them.
\end{abstract}

\section{Introduction}
Let $G$ be a finite $\pi$-separable group for a set $\pi$ of primes,
and let $\chi\in\Irr(G)$ be an irreducible complex character of $G$.
Suppose that $Q$ is a $\pi'$-subgroup of $G$ and $\delta\in\Irr(Q)$.
Following Cossey \cite{C2010}, we say that the pair $(Q,\delta)$ is a {\bf vertex pair} for $\chi$ if there exists a subgroup $U$ of $G$ such that $Q$ is a Hall $\pi'$-subgroup of $U$
and $\delta=(\psi_{\pi'})_Q$, where $\psi\in\Irr(U)$ is $\pi$-factored with $\psi^G=\chi$ and $\psi_{\pi'}$ is the $\pi'$-special factor of $\psi$.
In particular, if $U$ and $\psi$ are chosen
so that the pair $(U,\psi)$ is a normal nucleus of $\chi$ (see \cite{N2002} for details),
then $(Q,\delta)$ is known as a {\bf Navarro vertex} for $\chi$.
We remark that in general, all of the vertex pairs for $\chi$ need not be conjugate,
but the Navarro vertex for $\chi$ is unique up to conjugacy.
In this paper we are interested in {\bf linear vertex pairs},
that is, those vertex pairs $(Q,\delta)$ with $\delta$ a linear character.

Linear vertex pairs play an essential role in proving theorems about lifts of Brauer characters or more general Isaacs' $\pi$-partial characters; see, for example, \cite{C2007, C2008, C2009, CL2011, CLN2011, C2012, CL2012, N2002} and especially the review paper \cite{C2010} mentioned above. A crucial result is that in the case where $|G|$ is odd or $2\in\pi$,
all of the vertex pairs for a $\pi$-lift are linear and conjugate
(see Theorem 1.1 of \cite{C2008}, or Theorem 4.1 and Lemma 4.3 of \cite{CL2012}).
Recall that $\chi\in\Irr(G)$ is said to be a {\bf $\pi$-lift} if its restriction $\chi^0$ to the set $G^0$ of $\pi$-elements of the $\pi$-separable group $G$ is an irreducible $\pi$-partial character of $G$.
(For definition and properties of the $\pi$-partial characters, we refer to \cite{I2018};
here we mention that if $G$ is $p$-solvable and $\pi=p'$ is the complement of the prime $p$,
then the set $\I(G)$ of irreducible $\pi$-partial characters of $G$ is exactly the set $\IBr_p(G)$ of irreducible $p$-Brauer characters of $G$.)

The aim of the present paper is to explore the uniqueness of linear vertex pairs in the remaining case where $2\notin\pi$.
In order to study the behavior of a given $\pi$-lift with respect to a chain of normal subgroups,
Cossey and Lewis \cite{CL2011} introduced the notion of the inductive vertices
for $\mathcal N$-lifts, and showed that these vertex pairs are linear and `almost' unique up to conjugacy (Theorem 1.2 of that paper; see also Theorem 5.4 of \cite{C2010}).
Here $\mathcal N$ is a {\bf $\pi$-chain} of $G$,
which by definition is a chain of normal subgroups of $G$:
$$\mathcal N=\{1=N_0\le N_1\le \cdots \le N_{k-1}\le N_k=G\}$$
such that $N_i/N_{i-1}$ is either a $\pi$-group or a $\pi'$-group for $i=1,\dots,k$,
and $\chi\in\Irr(G)$ is an {\bf $\mathcal N$-lift} if every irreducible constituent of $\chi_N$ is a $\pi$-lift for all $N\in\mathcal N$.
Inspired by their work, we define the {\bf twisted vertex}
for any $\chi\in\Irr(G)$ to be $(Q,\epsilon\delta)$
by adding some permutation sign character $\epsilon$
to the original vertex pair $(Q,\delta)$ of $\chi$ (see Definition 2.1 below),
and we will study the uniqueness of such vertices.

\begin{thmA}
Let $G$ be a $\pi$-separable group with $2\notin\pi$, and
suppose that $\chi\in\Irr(G)$ is an $\mathcal{N}$-lift for some $\pi$-chain $\mathcal{N}$ of $G$. Then all of the linear twisted vertices for $\chi$ are conjugate in $G$.
\end{thmA}

This gives a positive answer to a question proposed by Cossey and Lewis
of whether or not the nilpotence assumption on the $\pi$-chain $\mathcal N$ appeared in
Theorem 5.4 of \cite{C2010} or Theorem 4.2 of \cite{CL2011} is really necessary.
We remark that the linear twisted vertices in Theorem A definitely exist
(see  Corollary \ref{basic} below).

Using Navarro vertices, we can present another condition sufficient to guarantee the uniqueness for our linear twisted vertices.

\begin{thmA}
Let $G$ be $\pi$-separable with $2\notin\pi$, and let $\chi\in\Irr(G)$ be a $\pi$-lift.
If $\chi$ has a linear Navarro vertex, then all of the linear twisted vertex pairs for $\chi$ are conjugate in $G$.
\end{thmA}

In \cite{WJ2022}, we will use Theorem B to study a conjecture of Cossey \cite{C2007}
on the number of lifts of a given irreducible Brauer character in characteristic two.
Observe that the assumption on Navarro vertices in Theorem B is trivially satisfied
if either $\chi$ has $\pi$-degree or $G$ has an abelian Hall $\pi'$-subgroup.

We now present an immediate consequence of both Theorem A and Theorem B.
Recall that in order to obtain canonical lifts for $\I(G)$ in a $\pi$-separable group $G$,
several important subsets of $\Irr(G)$ were constructed,
such as $\B(G)$, $\D(G)$, $\N(G)$ and $\B(G,{\mathcal N})$.
(See \cite{I1984, I1996, N2002, L2006} in turn for the relevant definitions and properties.)
We mention that all characters in $\N(G)$ satisfy the hypothesis in Theorem B,
and each member of $\B(G)$,
$\D(G)$ or $\B(G,{\mathcal N})$ is an  $\mathcal N$-lift for any $\pi$-chain $\mathcal{N}$ of $G$, so that Theorem A applies.

\begin{corA}
Let $\chi\in\Irr(G)$, where $G$ is a $\pi$-separable group with $2\notin\pi$.
If $\chi$ is any member of $\B(G)$, $\D(G)$, $\N(G)$, or $\B(G,\mathcal{N})$
for a $\pi$-chain $\mathcal{N}$ of $G$,
then $\chi$ has a linear vertex pair, and all of the linear twisted vertices for $\chi$ are conjugate in $G$.
\end{corA}

Most of our notation and terminology can be found in \cite{I1976} and \cite{I2018}.
In particular, if $N$ is a normal subgroup of a finite group $G$ and $\theta\in\Irr(N)$,
we always write $G_\theta$ for the inertia group of $\theta$ in $G$
and $\chi_\theta$ for the Clifford correspondent with respect to $\theta$
of any $\chi\in\Irr(G)$ lying over $\theta$.
Also, if a character $\chi$ is $\pi$-factored, that is, if $\chi=\alpha\beta$,
where $\alpha$ is $\pi$-special and $\beta$ is $\pi'$-special,
then $\alpha$ and $\beta$ are uniquely determined by $\chi$ by a theorem of Gajendragadkar (see Theorem 2.2 of \cite{I2018}),
so for notational convenience, we often use $\chi_{\pi}$ and $\chi_{\pi'}$ to denote the $\pi$-special factor $\alpha$ and the $\pi'$-special factor $\beta$ of $\chi$, respectively.

\section{Preliminaries}

We begin with the definition of twisted vertices mentioned in the introduction.

\begin{defn}
Let $\chi\in\Irr(G)$, where $G$ is a $\pi$-separable group with $2\notin\pi$,
and suppose that $Q$ is a $\pi'$-subgroup of $G$ and $\delta\in\Irr(Q)$.
We say that the pair $(Q,\delta)$ is a {\bf twisted vertex} for $\chi$
if there exists a subgroup $U$ of $G$ such that $Q$ is a Hall $\pi'$-subgroup of $U$
and $\delta=(\det((1_Q)^U))_Q(\psi_{\pi'})_Q$,
where $\psi\in\Irr(U)$ is $\pi$-factored with $\psi^G=\chi$.
If $(Q,\delta)$ is a twisted vertex pair for $\chi$ and $\delta$ is linear,
then  $(Q,\delta)$ is called a {\bf linear twisted vertex pair} for $\chi$.
\end{defn}

It is easy to see that $\epsilon=(\det((1_Q)^U))_Q$ is the permutation sign character of the action of $Q$ on the right cosets of $Q$ in $U$, so that $\epsilon$ is a sign character,
which means that its values are $\pm 1$.
For any subgroup $H$ of $G$, Isaacs \cite{I1986} defined the {\bf $\pi$-standard sign character}
$\delta_{(G,H)}$  (see also \cite{I2018}), and for convenience we list some relevant properties as follows.

\begin{lem}\label{sign}
Let $G$ be a $\pi$-separable group, where $2\notin\pi$.

{\rm (1)} If $Q$ is a Hall $\pi'$-subgroup of $G$, then $\delta_{(G,Q)}=(\det((1_Q)^G))_Q$.

{\rm (2)} If $Y\le X\le G$, then $\delta_{(G,Y)}=(\delta_{(G,X)})_Y\delta_{(X,Y)}$.

{\rm (3)} If $N\NM G$, then $\delta_{(G,N)}=1_N$, the principal character of $N$.
\end{lem}
\begin{proof}
(1) is Lemma 2.10 of \cite{I1986},
and for (2) and (3), see Theorem 2.5 of that paper or Lemma 2.33 of \cite{I2018}.
\end{proof}

By Lemma \ref{sign}(1),
we can modify Definition 2.1 by replacing the permutation sign character $(\det((1_Q)^U))_Q$ with the $\pi$-standard sign character $\delta_{(U,Q)}$,
and obtain a more convenient definition of twisted vertices.
From now on, we will always use this new definition.

We need some elementary properties of $\mathcal{N}$-lifts,
which appeared implicitly in \cite{CL2011}.

\begin{lem}\label{pi-deg}
Let $\chi\in\Irr(G)$ be an $\mathcal{N}$-lift,
where $\mathcal N$ is a $\pi$-chain of a $\pi$-separable group $G$,
and let $N\in\mathcal{N}$.
Suppose that $\theta\in\Irr(N)$ is a constituent of $\chi_N$. Then the following hold.

{\rm (1)} $\chi_\theta\in\Irr(G_\theta)$ is an $\mathcal{N}_\theta$-lift, where $\mathcal{N}_\theta=\{M\cap G_\theta|M\in\mathcal{N}\}$ is a $\pi$-chain for $G_\theta$.

{\rm (2)} If $\theta$ is $\pi$-factored, then $\theta$ has $\pi$-degree.
\end{lem}
\begin{proof}
(1) By definition, it is easy to see that $\mathcal{N}_\theta$ is a $\pi$-chain for $G_\theta$.
To prove that $\chi_\theta$ is an $\mathcal{N}_\theta$-lift,
we fix an arbitrary $M\in\mathcal N$ and $\eta\in\Irr(M\cap G_\theta)$ lying under $\chi_\theta$,
and we want to show that $\eta$ is a $\pi$-lift.
In this case, we have two possibilities: either $M\le N$ or $N\le M$.

If $M\le N$, then $M\cap G_\theta=M$,
and since $\chi$ is an $\mathcal N$-lift and $\eta\in\Irr(M)$ lies under $\chi$, we know that $\eta$ is a $\pi$-lift. We may therefore assume that $M\ge N$,
and then $M\cap G_\theta=M_\theta$ is the inertia of $\theta$ in $M$.
Now $\eta$ lies over $\theta$, so the Clifford correspondence implies that $\eta^M$ is irreducible,
which clearly lies under $\chi$. It follows that $\eta^M$ is a $\pi$-lift,
and thus $(\eta^0)^M=(\eta^M)^0\in\I(M)$.
This shows that $\eta^0$ is irreducible, i.e., $\eta$ is also a $\pi$-lift, as wanted.

(2) We proceed by induction on $|N|$.
If $N=1$, there is nothing to prove, so we can assume that $N>1$,
and then there exists $M\in\mathcal N$ such that
$M<N$ with $N/M$ either a $\pi$-group or a $\pi'$-group.
Let $\eta\in\Irr(M)$ lie under $\theta$. Then $\eta$ is also $\pi$-factored,
and the inductive hypothesis guarantees that $\eta(1)$ is a $\pi$-number.
Since $\theta(1)$ divides $|N:M|\eta(1)$, the result follows in the case where $N/M$ is a $\pi$-group.

Now suppose that $N/M$ is a $\pi'$-group.
Then, since both $\theta$ and $\eta$ are $\pi$-factored, we have $(\theta_\pi)_M=\eta_\pi$ by Clifford's theorem, and thus $\eta_\pi$ is invariant in $N$.
Note that both $\theta$ and $\eta$ are $\pi$-lifts because $\chi$ is assumed to be an $\mathcal N$-lift, so $\eta^0\in\I(M)$ lies under $\theta^0\in\I(N)$.
But $\eta(1)$ is a $\pi$-number, so $\eta_{\pi'}$ is linear,
and we have $\eta^0=(\eta_{\pi})^0(\eta_{\pi'})^0=(\eta_{\pi})^0$. It follows that $\eta^0$ is $N$-invariant, and by Clifford's theorem for $\pi$-partial characters (see Corollary 5.7 of \cite{I2018}), we deduce that $(\theta^0)_M=\eta^0$.
This shows that $\theta(1)=\eta(1)$, which is a $\pi$-number, as required.
\end{proof}

Although the following result is known (see Theorem 1 or Lemma 2.4 of \cite{L2010}),
as a simple application of Lemma \ref{pi-deg},
we present an alternative proof for completeness.

\begin{cor}\label{basic}
Let  $\mathcal N$ be a $\pi$-chain of a $\pi$-separable group $G$,
and let $\chi\in\Irr(G)$ be an $\mathcal{N}$-lift.
Then $\chi=\psi^G$ for some $\pi$-factored character $\psi$ of a subgroup $U\le G$,
where $\psi(1)$ is a $\pi$-number. In particular, $\chi$ has a linear twisted vertex.
\end{cor}
\begin{proof}
Suppose first that $\chi$ is not $\pi$-factored.
Let $N\in\mathcal N$ be maximal such that the constituents of $\chi_N$ are $\pi$-factored,
and let $\theta\in\Irr(N)$ lie under $\chi$.
We claim that $G_\theta<G$.
To see this, note that $N<G$, so there exists $M\in\mathcal N$ such that $M/N>1$ is a $\pi$-group or a $\pi'$-group. If the $\pi$-factored character $\theta$ is $G$-invariant, then both $\theta_\pi$ and $\theta_{\pi'}$ are invariant in $G$, and by Lemma 4.2 of \cite{I2018},
every member of $\Irr(M|\theta)$ is $\pi$-factored.
In particular, each irreducible constituent $\eta$ of $\chi_M$ lying over $\theta$ is $\pi$-factored,
and this contradicts the maximality of $N$.
Thus $\theta$ cannot be invariant in $G$, as claimed.

Now we work by induction on $|G|$.
By Lemma \ref{pi-deg}(1) and the inductive hypothesis applied in $G_\theta$,
we have $\chi_\theta=\psi^{G_\theta}$, where $\psi$ is a $\pi$-factored character of some subgroup $U\le G_\theta$ and $\psi(1)$ is a $\pi$-number.
Since $\chi=(\chi_\theta)^G=\psi^G$, the result follows in this case.

In the remaining case, $\chi$ is $\pi$-factored.
Then $\chi(1)$ is a $\pi$-number by Lemma \ref{pi-deg}(2),
and the proof is complete by taking $(U,\psi)=(G,\chi)$.
\end{proof}

\section{The uniqueness of twisted vertices}

In this section we will prove our main theorems.
The following `replacement lemma' is key to our purpose,
which might be of independent interest on its own.
Note that the assumption that $2\notin\pi$ is essential,
and the lifts of $\pi$-partial characters are not involved
(compare with Lemma 4.8 of \cite{CL2012}).

\begin{lem}\label{replace}
Let $\chi\in\Irr(G)$, where $G$ is $\pi$-separable with $2\notin\pi$,
and suppose that $\chi=\psi^G$, where $\psi\in\Irr(U)$ is $\pi$-factored and $U\le G$.
Assume that $\chi_N$ has a $\pi$-factored irreducible constituent for some $N\NM G$.
Write $\xi=\psi^{NU}$.
If $|NU:U|$ is a $\pi$-number, then $\xi$ is $\pi$-factored,
and its $\pi$-special factor $\xi_\pi$ and $\pi'$-special factor $\xi_{\pi'}$
satisfy $\xi_{\pi}=(\delta_{(NU,U)}\psi_{\pi})^{NU}$ and $(\xi_{\pi'})_U=\delta_{(NU,U)}\psi_{\pi'}$, respectively.
\end{lem}

\begin{proof}
We proceed by double induction, first on $|G:U|$ and then on $|N|$.
Note that all irreducible constituents of $\chi_N$ are conjugate in $G$ by Clifford's theorem,
and since $\xi^G=\psi^G=\chi$, it follows that $\xi$ lies under $\chi$ and thus each irreducible constituent of $\xi_N$ is also $\pi$-factored.
So it is no loss to assume that $G=NU$.

If $N=1$, then $G=NU=U$, and there is noting to prove.
So we assume that $N>1$, and take $N/K$ to be a chief factor of $G$.
Write $D=N\cap U$, so that $D\le KD\le N$.
If $KD=N$, then $G=(KD)U=KU$, and since $K<N$
and each irreducible constituent of $\chi_K$ is $\pi$-factored,
the result follows by the inductive hypothesis applied in the group $N$.

We can assume, therefore, that $KD<N$.
Write $X=KD$. Then $|N:X|>1$ divides $|G:U|$, which is a $\pi$-number,
so $N/K$ is a $\pi$-group.
Since we are assuming that $2\notin\pi$,
it follows that the chief factor $N/K$ has odd order and hence is abelian
(using the odd-order theorem).
This shows that $X\NM N$, which forces $X=K$, and thus $D\leq K<N$.
In this case, we have two possibilities: either $D<K$ or $D=K$.

Suppose first that $D<K$. Let $Y=KU$ and $\zeta=\psi^Y$.
Then $U<Y<G$ and $\chi=\zeta^G$, so that each irreducible constituent of $\zeta_K$ lies under $\chi$
and thus is $\pi$-factored.
Since $|Y:U|<|G:U|$, the inductive hypothesis (with $Y$ in place of $G$)
implies that $\zeta$ is $\pi$-factored with $\zeta_\pi=(\delta_{(Y,U)}\psi_{\pi})^Y$ and $(\zeta_{\pi'})_U=\delta_{(Y,U)}\psi_{\pi'}$.
Again, we have $|G:Y|<|G:U|$, so the inductive hypothesis guarantees that
$\chi$ is $\pi$-factored with $\chi_\pi=(\delta_{(G,Y)}\zeta_\pi)^G$ and $(\chi_{\pi'})_Y=\delta_{(G,Y)}\zeta_{\pi'}$.
Combining these two cases and using the fact that
$(\delta_{(G,Y)})_U\delta_{(Y,U)}=\delta_{(G,U)}$ (see Lemma \ref{sign}(2)),  we obtain
$$\chi_\pi=(\delta_{(G,Y)}(\delta_{(Y,U)}\psi_{\pi})^Y)^G=
((\delta_{(G,Y)})_U\delta_{(Y,U)}\psi_{\pi})^Y)^G=(\delta_{(G,U)}\psi_{\pi})^G$$
and $(\chi_{\pi'})_U=(\delta_{(G,Y)})_U\delta_{(Y,U)}\psi_{\pi'}=
\delta_{(G,U)}\psi_{\pi'}$, as required.

Now suppose that $D=K\NM G$.
Let $\eta\in\Irr(D)$ lie under $\psi$,
and choose some $\theta\in\Irr(N|\eta)$ lying under $\chi$.
Then $\eta$ is $\pi$-factored because $\psi$ is $\pi$-factored and $D\NM U$. Note that $\theta_{\pi'}$ restricts irreducibly to $D$ (see Theorem 2.10 of \cite{I2018}),
and by the normality of $D$ in $N$, it is easy to see that $(\theta_{\pi'})_D=\eta_{\pi'}$.
Also, since $\theta_{\pi'}$ is the unique $\pi'$-special character of $N$ lying over $\eta_{\pi'}$,  both characters have the same inertia group in $G$.
Observe that $D\NM G$ is contained in the kernel of $\delta_{(G,U)}$ (Lemma \ref{sign}(3))
and that $\psi_{\pi'}$ lies over $\eta_{\pi'}$, so  $\delta_{(G,U)}\psi_{\pi'}$ also lies over $\eta_{\pi'}$ and extends to $G$ by Lemma 2.11 of \cite{I2018}.
Furthermore, we can deduce from Corollary 3.15 of \cite{I2018}
(with the roles of $\pi$ and $\pi'$ interchanged) that
there exists a $\pi'$-special character $\beta$ of $G$ such that
$\beta_U=\delta_{(G,U)}\psi_{\pi'}$.
Write $\alpha=(\delta_{(G,U)}\psi_{\pi})^G$. Then we have
$$\chi=\psi^G=(\delta_{(G,U)}\psi_{\pi}\beta_U)^G=\alpha\beta,$$
so $\alpha$ is irreducible, and by Theorem 7.25 of \cite{I2018}, we see that $\alpha$ is $\pi$-special.
This shows that $\chi$ is $\pi$-factored with $\chi_\pi=\alpha$ and $\chi_{\pi'}=\beta$,
and the proof is now complete.
\end{proof}

The following technical result will be used twice later.
\begin{cor}\label{simple}
Let $G$ be $\pi$-separable with $2\notin\pi$, and let $\chi\in\Irr(G)$.
If $\chi$ is $\pi$-factored and $\chi(1)$ is a $\pi$-number,
then all of the twisted vertices for $\chi$ are conjugate in $G$.
\end{cor}

\begin{proof}
Let  $(Q,\delta)$ be a twisted vertex for $\chi$.
By definition, $Q$ is a Hall $\pi'$-subgroup of some subgroup $U\le G$,
and there is a $\pi$-factored character $\psi$ of $U$ such that $\psi^G=\chi$
and $\delta=\delta_{(U,Q)}(\psi_{\pi'})_Q$.
Since $\chi$ has $\pi$-degree, we see that $|G:U|$ is a $\pi$-number and hence $Q$ is a Hall $\pi'$-subgroup of $G$. Also, by Lemma \ref{replace} (with $N=G$),
we know that $(\chi_{\pi'})_U=\delta_{(G,U)}\psi_{\pi'}$, and thus
$$\delta=\delta_{(U,Q)}(\psi_{\pi'})_Q=\delta_{(U,Q)}(\delta_{(G,U)})_Q(\chi_{\pi'})_Q
=\delta_{(G,Q)}(\chi_{\pi'})_Q.$$
Since the Hall $\pi'$-subgroup $Q$ is uniquely determined up to $G$-conjugacy,
it is clear that the twisted vertex $(Q,\delta)=(Q,\delta_{(G,Q)}(\chi_{\pi'})_Q)$
is also unique up to conjugation in $G$, as desired.
\end{proof}

We consider the Clifford correspondents for $\pi$-factored characters.

\begin{lem}\label{Clifford}
Let $G$ be a $\pi$-separable group with $2\notin\pi$,
and let $\chi\in\Irr(G)$ be $\pi$-factored with $\chi(1)$ a $\pi$-number.
Assume that $N\NM G$ and $\theta\in\Irr(N)$ lie under $\chi$.
Let $\psi\in\Irr(T)$ be the Clifford correspondent of $\chi$ with respect to $\theta$,
where $T=G_\theta$ is the inertia group of $\theta$ in $G$.
Then $\psi$ is $\pi$-factored, and also $\chi_\pi=(\delta_{(G,T)}\psi_\pi)^G$
and $(\chi_{\pi'})_T=\delta_{(G,T)}\psi_{\pi'}$.
\end{lem}
\begin{proof}
By the hypothesis on $\chi$, it is easy to see that $\theta$ is $\pi$-factored.
Write $\lambda=\chi_{\pi'}$.
Then $\lambda$ is linear because $\chi(1)$ is a $\pi$-number,
and thus $\lambda_N=\theta_{\pi'}$.
In particular, $\theta_{\pi'}$ is invariant in $G$,
and we have $T=G_{\theta_\pi}\cap G_{\theta_{\pi'}}=G_{\theta_{\pi}}$.
Now $\chi_\pi$ lies over $\theta_\pi$, and we let $\xi\in\Irr(T)$
be the Clifford correspondent of $\chi_\pi$ with respect to $\theta_\pi$.
Then $\xi^G=\chi_\pi$, so $(\xi\lambda_T)^G=\xi^G\lambda=\chi$.
Clearly $\xi\lambda_T$ lies over $\theta_\pi\theta_{\pi'}=\theta$,
and hence $\psi=\xi\lambda_T$.
By Theorem 2.37 of \cite{I2018}, we see that $\delta_{(G,T)}\xi$ is $\pi$-special.
On the other hand, since $\lambda$ is a $\pi'$-special linear character of $G$,
it follows that $\lambda_T$ is also $\pi'$-special.
Note that $\delta_{(G,T)}$ is always $\pi'$-special, and so $\delta_{(G,T)}\lambda_T$ is $\pi'$-special. This shows that $\psi$ is $\pi$-factored with
$\psi_\pi=\delta_{(G,T)}\xi$ and $\psi_{\pi'}=\delta_{(G,T)}\lambda_T$,
and thus $\chi_\pi=\xi^G=(\delta_{(G,T)}\psi_\pi)^G$
and $(\chi_{\pi'})_T=\lambda_T=\delta_{(G,T)}\psi_{\pi'}$,
as desired.
\end{proof}

Now we are ready to prove Theorems A and B.
Since the arguments for both results are similar,
we combine them into a single theorem as follows.

\begin{thm}
Let $G$ be $\pi$-separable with $2\notin\pi$, and let $\chi\in\Irr(G)$ be a $\pi$-lift.
Assume one of the following conditions.

{\rm (a)} $\chi$ is an $\mathcal{N}$-lift for some $\pi$-chain $\mathcal{N}$ of $G$.

{\rm (b)} $\chi$ has a linear Navarro vertex.\\
Then all of the linear twisted vertices for $\chi$ are conjugate in $G$.
\end{thm}

\begin{proof}
We proceed by induction on $|G|$.
Fix an arbitrary linear twisted vertex $(Q,\delta)$ for $\chi$,
so by definition, $Q$ is a Hall $\pi'$-subgroup of some subgroup $U\le G$,
and there is a $\pi$-factored character $\psi$ of $U$ such that $\psi^G=\chi$
and $\delta=\delta_{(U,Q)}(\psi_{\pi'})_Q$.
Note that $\psi$ has $\pi$-degree because $\psi_{\pi'}(1)=\delta(1)=1$.

First, suppose that $\chi$ is $\pi$-factored.
By Corollary \ref{simple}, it suffices to prove that $\chi(1)$ is a $\pi$-number.
In case (a), this is true by Lemma \ref{pi-deg}(2),
and in case (b), since $(G,\chi)$ is itself a normal nucleus for $\chi$,
it follows by the definition of Navarro vertices that $\chi$ has $\pi$-degree.

Now suppose that $\chi$ is not $\pi$-factored.
We need to pick out an appropriate normal subgroup $N$ of $G$
such that each irreducible constituents of $\chi_N$ is $\pi$-factored but not invariant in $G$.
In case (a), we can choose $N\in\mathcal N$ maximal with the property that
an irreducible constituent $\theta$ is $\pi$-factored,
and so the same reasoning as we did in the beginning of the proof of Corollary \ref{basic},
yields $G_\theta<G$.
In case (b), we let $N$ to be the unique normal subgroup of $G$ such that
some $\theta\in\Irr(N)$ lying under $\chi$ is $\pi$-factored,
and by the construction of Navarro vertices, we have $G_\theta<G$, too.
Observe that in both cases the Clifford correspondent $\chi_\theta$ of $\chi$ with respect to $\theta$ also satisfies the assumption on $\chi$ (by Lemma \ref{pi-deg}(1) in case (a)).
Since $G_\theta<G$, the inductive hypothesis implies that
all the linear twisted vertices for $\chi_\theta$ are conjugate in $G_\theta$,
and we will complete the proof by showing that $(Q,\delta)$ is $G$-conjugate to some linear twisted vertex for $\chi_\theta$.

We may assume that $N\le U$.
To see this, note that $\theta(1)$ is a $\pi$-number in both case (a) by Lemma \ref{pi-deg}(2)
and case (b) by the construction of Navarro vertices.
It follows that $\theta^0\in\I(N)$ is an irreducible constituent of $(\chi^0)_N$ with $\pi$-degree.
Recall that $\chi=\psi^G$, so $\chi^0=(\psi^0)^G$ and $\psi^0\in\I(U)$, and Lemma 5.21 of \cite{I2018} applies.
We conclude that $|NU:U|$ is a $\pi$-number, and thus $Q$ is a Hall $\pi'$-subgroup of $NU$.
Write $\xi=\psi^{NU}$, so that $\xi^G=\chi$. By Lemma \ref{replace},
$\xi$ is $\pi$-factored with $(\xi_{\pi'})_U=\delta_{(NU,U)}\psi_{\pi'}$, and we have
$$\delta=\delta_{(U,Q)}(\psi_{\pi'})_Q=\delta_{(U,Q)}(\delta_{(NU,U)})_Q(\xi_{\pi'})_Q
=\delta_{(NU,Q)}(\xi_{\pi'})_Q.$$
This shows that $(Q,\delta)$ is also a linear twisted vertex for $\xi$,
and we can thus replace $U$ by $NU$ and $\psi$ by $\xi$.
So it is no loss to assume that $N\le U$, as wanted.

Furthermore, we may assume that $\theta$ lies under $\psi$ by replacing $\theta$ with a conjugate if necessary.
For notational simplicity, we write $T=G_\theta$ and $S=U_\theta$ for the inertia groups of $\theta$ in $G$ and in $U$,
and let $\alpha=\chi_\theta$ and $\beta=\psi_\theta$ be the Clifford correspondents of $\chi$ and $\psi$ with respect to $\theta$, respectively.
By Lemma \ref{Clifford}, we know that $\beta\in\Irr(S)$ is $\pi$-factored with $(\psi_{\pi'})_S=\delta_{(U,S)}\beta_{\pi'}$.
Now let $Q_1$ be a Hall $\pi'$-subgroup of $S$ and let $\delta_1=\delta_{(S,Q_1)}(\beta_{\pi'})_{Q_1}$.
Since $\psi=\beta^U$ and $\alpha=\beta^T$ by the Clifford correspondence,
it follows by definition that $(Q_1,\delta_1)$ is a common linear twisted vertex for $\psi$
and $\alpha$.
But $\psi$ is $\pi$-factored with $\pi$-degree,
so Corollary \ref{simple} implies that $(Q,\delta)$ and $(Q_1,\delta_1)$ are conjugate in $U$. This shows that $(Q,\delta)$ is conjugate to a linear twisted vertex for $\alpha=\chi_\theta$, and the proof is now complete.
\end{proof}

\section*{Acknowledgement}

This work was supported by the NSF of China (No. 12171289) and the NSF of Shanxi Province (Nos. 20210302123429, 20210302124077).



\end{document}